\documentclass{amsart}
\usepackage[utf8]{inputenc}
\usepackage{amsfonts}
\usepackage{amsmath}
\numberwithin{equation}{section}
\usepackage{amssymb}
\usepackage{float}
\usepackage{tikz-cd}
\usepackage{amsthm}
\usepackage{enumerate}
\usepackage{tikz}
\usepackage{graphicx}
\usepackage{hyperref}
\usepackage[capitalise]{cleveref}
\usepackage{slashed}  
\usepackage{upgreek}
\usepackage[sorting=nyt]{biblatex}
\usepackage[toc]{appendix}

\usepackage[top=3cm,bottom=3cm,left=2.5cm,right=2.5cm,headsep=10pt,letterpaper]{geometry} 

\theoremstyle{plain}
\newtheorem{theorem}{Theorem}[section]
\newtheorem{cor}[theorem]{Corollary}
\newtheorem{prop}[theorem]{Proposition}
\newtheorem{lemma}[theorem]{Lemma}

\theoremstyle{definition}
\newtheorem{definition}[theorem]{Definition}
\newtheorem{example}[theorem]{Example}
\theoremstyle{remark}
\newtheorem{remark}{Remark}[section]

\newtheorem*{claim}{Claim}

\def\RR {\mathbb{R}}
\def\KK {\mathbb{K}}
\def\CC {\mathbb{C}}
\def\HH {\mathbb{H}}
\def\OO {\mathbb{O}}
\def\SS {\mathbb{S}}

\def\PP {\mathbb{P}}

\newcommand{\abs}[1]{\left\lvert #1 \right\rvert}  

\DeclareMathOperator{\sech}{sech}
\DeclareMathOperator{\csch}{csch}

\begin{document}

\title{On the PPW Conjecture for Hopf-Symmetric Sets in Non-compact Rank One Symmetric Space}
\author{Yusen Xia}
\address{Department of Mathematics, University of California, Santa Barbara, CA 93106}
\email{yusen@ucsb.edu}


\begin{abstract}
    In this paper, we proved that for a bounded Hopf-symmetric domain $\Omega$ in a non-compact rank one symmetric space $M$, the second Dirichlet eigenvalue $\lambda_2(\Omega)\leq \lambda_2(B_1)$ where $B_1$ is a geodesic ball in $M$ such that $\lambda_1(\Omega) = \lambda_1(B_1)$. This generalizes the work of Ashbaugh $\&$ Benguria, Benguria $\&$ Linde for bounded domains in constant curvature spaces. 
\end{abstract}

\maketitle
\tableofcontents

\section{Introduction}

\subsection{Statement of the main theorem}
 Given an open bounded domain $\Omega$ in a Riemannian manifold $M$, we consider the PDE \[\Delta u+\lambda u=0 \text { in } \Omega, \quad u=0 \text { on } \partial \Omega,\] where $\Delta$ is the Laplace-Beltrami operator. The spectrum consists of a discrete sequence of eigenvalues $0<\lambda_1<\lambda_2 \leq \lambda_3 \cdots \rightarrow \infty$. Physically, the eigenvalues $\lambda_i$ correspond to the natural frequencies of vibration of a drumhead with shape $\Omega$.The celebrated question, “Can one hear the shape of a drum?” refers to the problem of determining the geometry of $\Omega$ from the spectral data $\lambda_i(\Omega)$.
\par In the 1990s, Ashbaugh and Benguria\cite{ashbaugh1992sharp} proved the Payne-Pólya-Weinberger (PPW) conjecture, which asserts that among all open bounded regions in $\RR^n$, balls maximize the ratio $\frac{\lambda_2}{\lambda_1}$ of the first two Dirichlet eigenvalues. Physically, this implies that one can distinguish whether a drum is an Euclidean ball purely by comparing the frequencies of its first two tones. They subsequently extended this result to the hemispheres\cite{ashbaugh2001sharp}, though a reformulation is necessary: the ratio of the first two Dirichlet eigenvalues is scale-invariant in $\RR^n$, but not in $\SS^n$. The solution is that, instead of maximizing the ratio $\frac{\lambda_2}{\lambda_1}$, geodesic balls in $\SS^n$ maximize the second eigenvalue $\lambda_2$ among regions with fixed $\lambda_1$. Benguria and Linde further extended this result to hyperbolic space $\HH^n$\cite{benguria2007second}. For spaces of nonconstant curvature, Nick Edelen \cite{edelen2017ppw} proved a gap comparison inequality for $\lambda_2-\lambda_1$ for manifolds with sectional curvature upper bound and Ricci curvature lower bound. His result has an additional curvature-dependent factor. For a bit more history, see \cref{brief history}. 
\par In this paper we prove an extension of the PPW conjecture to rank one symmetric spaces (ROSS) of noncompact type:
\begin{theorem}\label{MainTheorem}
    Let $M$ be one of the non-compact rank one symmetric space $\CC\HH^n, \HH\HH^n$ with $n>1$ or $\OO\HH^2$ . Let $\Omega \subset M$ be a bounded open set. Let $\lambda_i(\Omega)$ be its first and second Dirichlet eigenvalues for $i=1,2$, respectively. Let $B_1$ be the geodesic ball centered at some point $p \in M$ and $\lambda_1(B_1)= \lambda_1(\Omega)$. Assume $\Omega$ is Hopf symmetric. Then \[\lambda_2(\Omega) \leq \lambda_2\left(B_1\right)\] with equality iff $\Omega$ is a geodesic ball.
\end{theorem}
For the precise definition of Hopf symmetric sets, see \cref{DefHopfSymmetricSets}. In spaces of constant curvature, the Hopf symmetry condition is automatically satisfied, and thus the classical result in $\HH^n$ is recovered. We restrict our domain $\Omega$ to be hopf-symmetric because the exact solution to the isoperimetric problem in noncompact rank-one symmetric spaces (ROSS) is known only in this setting. This restriction allows us to invoke a Faber–Krahn type isoperimetric inequality, as well as Chiti’s comparison theorem. For details, see \cref{FaberKrahn} and \cref{ChitiCompThm}.
\begin{remark}
 The isoperimetric inequality for the general set (beyond Hopf-symmetric) in $\KK\HH^n$ is conjectured to be true (see, for example, \cite[335]{GromovMisha2007MSfR} and \cite[Question 112, p. 339]{Berger2003RiemGeo}). Once the isoperimetric problem in noncompact rank-one symmetric spaces is fully resolved, the Hopf-symmetric assumption in our result can be removed.
\end{remark}

The strategy to prove the PPW conjecture is similar in constant and non-constant-curvature spaces\cite{ashbaugh1992sharp,ashbaugh2001sharp,benguria2007second,edelen2017ppw}. Namely, one chooses the test function in the Rayleigh characterization to be the quotient of the first two Dirichlet eigenfunctions of a suitable geodesic ball, then applies spherical rearrangements to transfer the integrand from $\Omega$ to a ball $\Omega^*$ with the same volume. To move from $\Omega^*$ to a ball $B_1$ with the same $\lambda_1$ as $\Omega$, we need the monotonicity of the integrand. This is the major contribution of our work. To do an argument similar to \cite{benguria2007second} but in rank one symmetric spaces, several extra terms occur which makes the calculations much more involved. %

\par Finally, we make several remarks on a possible extension to compact ROSS such as $\CC\mathbb{P}^n$ etc.
\begin{remark}\label{Remark:ExtensionToCompactROSS}
    \begin{enumerate}
    \item We define the auxiliary function $q$ as $q=r\frac{G'}{G}$ (see \cref{DefinitionOfQ}). However, a similar argument in the compact case may be more complicated. See \cref{remark:CompactRossQ}.
    \item For the isoperimetric inequality to hold, one may have to restrict the volume of the domain $\Omega$ to be very small. See \cite{VianaCelso2023Iavp} for the case of $\RR\mathbb{P}^n$. The isoperimetric regions are not only geodesic balls, but also tubes around $\RR\PP^k\subset \RR\PP^n$.
\end{enumerate}
\end{remark}

\subsection{A brief history of isoperimetric inequalities for eigenvalues of the Laplacian}\label{brief history}
The study of isoperimetric inequalities for eigenvalues of the Laplacian has a long history and has important applications in mathematics and physics. We briefly mention some milestones in this field. For a comprehensive introduction, we refer to the lecture notes by Ashbaugh and Benguria\cite{ashbaugh2007isoperimetric}, and Benguria et al.\cite{benguria2012isoperimetric}. 
\par One of the early triumphs was the Faber-Krahn inequality, which says that for a planar bounded domain $\Omega$, the first Dirichlet eigenvalue is minimized by a ball of the same volume. The basic technique to prove Faber-Krahn inequality is to use the classical isoperimetric inequality to show that spherical-decreasing rearrangement decreases the Dirichlet energy for $f\in H_0^1(\Omega)$. See \cref{section: Preliminaries} for more information on rearrangements. This inequality has been generalized to $\RR^n$, $\SS^n$ and $\HH^n$. 
\par On the first non-zero Neumann eigenvalue $\mu_1$, we have the Szegö-Weinberger inequality, which says that for a planar bounded domain $\Omega$, the first Dirichlet eigenvalue is maximized by a ball of the same volume. In 1954 Szegö first proved this inequality for planar domains using conformal mappings, and then Weinberger extended this result to $\RR^n$ with a new method in 1956. The major technique is to apply the Rayleigh characterization to sum of the $n$-folded first Neumann eigenfunction of the ball. This inequality has been generalized to domains contained in hemispheres and $\HH^n$ by Ashbaugh and Benguria\cite{ashbaugh1995sharp}. For domains in rank one symmetric spaces, Aithal and Santhanam\cite{aithal1996sharp} proved similar results.
\par The first non-zero Neumann eigenvalue is actually the second eigenvalue since the first Neumann eigenvalue is zero. Inspired by this fact, we look at the second Dirichlet eigenvalue. In the 1950s, Payne, Pólya and Weinberger (PPW) showed that for an open bounded domain $\Omega\subset\RR^2$, the ratio of the first two Dirichlet eigenvalues is bounded by 3\cite{payne1955quotient,payne1956ratio}. They also conjectured that this ratio is maximized by open disks. This conjecture was resolved by Ashbaugh and Benguria in $\Omega\subset\RR^n$. We refer to Section 1.1 for recent progress in this direction. 


\subsection{Organization of this paper}

After presenting preliminary results in \cref{section: Preliminaries}, we examine the first two Dirichlet eigenvalues and corresponding eigenfunctions of geodesic balls in rank-one symmetric spaces in \cref{Section: First two Dirichlet of balls}. \cref{section: Main thoerem proof} is devoted to the proof of the main theorem, with the exception of a key monotonicity proposition (see \cref{Monotonicity Lemma}), which is established in \cref{Section: MonotonicityLemma}. Some of the more technical calculations are deferred to the appendix.

\subsection*{Acknowledgment}
The author is grateful to his Ph.D. advisor Professor Guofang Wei for her invaluable support, fruitful discussions and suggestions, and constant encouragement. We would like to express our sincere gratitude to the referee for the careful reading of our manuscript and the helpful comments. These suggestions have greatly improved the clarity and quality of the paper. This research is partially supported by NSF DMS 2403557.

\section{Preliminaries}\label{section: Preliminaries}
In this section, we summarize the basic set-up and the tools that we use throughout this paper.

\subsection{Laplace operator in  rank one symmetric space}
We follow the notations in \cite{aithal1996sharp}. Let $(M^n,g)$ denote any one of the following rank one symmetric space (ROSS) of either compact type or non-compact type: the complex projective (hyperbolic resp.) space $\CC\PP^n$ ($\CC\HH^n$, resp.), the quaternionic projective (hyperbolic, resp.) space $\HH\PP^n$($\HH\HH^n$, resp) and the octonionic Caylay plane $\OO\PP^2$ or $\OO\HH^2$. Let $\KK\in \{\CC,\HH,\OO\}$ and $k= dim_{\RR} \KK$. Note the dimension of $M^n$ in this notation is $nk$. 
\par Let $v \in T_p M$ be a unit tangent vector and $\gamma_v(r)$ be the geodesic with $\gamma_v(0)=p$ and $\gamma_v^{\prime}(0)=v$. Let us denote by $J(v, r)$ the Riemannian volume density function along $\gamma_v(r)$. Since $M$ is a rank-1 symmetric space $J(v, r)$ is independent of $v$ and we write it as $J(r)$. For ROSS, the volume density is given by \[\text{Noncompact type: }J(r)=\sinh ^{k n-1} r \cosh ^{k-1} r.\]\[\text{Compact type: }J(r)=\sin ^{k n-1} r \cos ^{k-1} r.\]
The Laplacian of $M$ is given by (in geodesic polar coordinates)
\[\Delta_M=\frac{\partial^2}{\partial r^2}+H(r) \frac{\partial}{\partial r}+\Delta_{S(r)}\]
where 
\begin{equation}
    \text{Noncompact type: } H(r)=\frac{J'(r)}{J(r)}=(k n-1) \operatorname{coth} r+(k-1) \tanh r
\end{equation}
\[\text{Compact type: } H(r)=\frac{J'(r)}{J(r)}=(k n-1) \operatorname{cot} r-(k-1) \tan r \] is the mean curvature and $\Delta_{S(r)}$ is the Laplacian, respectively, of the geodesic sphere of radius $r$. As in \cite{aithal1996sharp}, the first eigenvalue of the geodesic sphere of radius $r$ is 
\begin{equation}\label{Lambda1Sr}
    \lambda_1(S(r))=\left(\frac{k n-1}{\sinh ^2 r}-\frac{k-1}{\cosh ^2 r}\right)\quad \text{for noncompact type}
\end{equation}
and  \[\lambda_1(S(r))=\left(\frac{k n-1}{\sin ^2 r}+\frac{k-1}{\cos ^2 r}\right)\quad \text{for compact type}\] with corresponding eigenfunctions being the coordinate functions $X_i, i =1,2,...,n$. Note $\lambda_1(S(r))= -H'(r)$ and is strictly decreasing for all $r>0$ in the noncompact case; and for $r\in (0,\pi/4]$ in the compact case.

\subsection{Spherical rearrangement}
For a set $\Omega\subset M$, the spherical rearrangement of $\Omega$ is a geodesic ball $\Omega^*$ such that $\operatorname{vol}(\Omega^*)=\operatorname{vol}(\Omega)$. For a positive function $f:\Omega \to \RR_+$, its spherical rearrangements is given by \cite[Definition 2.0.1]{edelen2017ppw}
\begin{definition}
        Define $f^*: \Omega^* \rightarrow \mathbb{R}^{+}$to be the radial function such that for any $t$,
$$
\operatorname{Vol}\{f \geq t\}=\operatorname{Vol}\{f^* \geq t\} .
$$
and $\{f^*>t\}$ is a ball.
where $\Omega^*$ is the geodesic ball with same volume as $\Omega$. Similarly, the spherical increasing rearrangement of $f$ is $f_*: \Omega^* \rightarrow \mathbb{R}^{+}$to be the radial function such that for any $t$,
$$
\operatorname{Vol}(\Omega)- \operatorname{Vol}\{f \geq t\}=\operatorname{Vol}\{f_* \geq t\} .
$$
\end{definition}
\begin{prop}[Properties of spherical rearrangement]
The following facts hold:
    \begin{itemize}
        \item $f^*$ is decreasing; $f_*$ is increasing.
        \item  $\|f_*\|_p$ =$\|f^*\|_p$ = $\|f\|_p$ for all $p\geq1$. In particular rearrangement preserves $L^2$ norm.
        \item If $f\in H_0^1 (\Omega)$, then $f^*\in H_0^1 (\Omega^*) $.
        \item $\int_{\Omega^*} f_*g^* \leq \int_{\Omega} fg \leq \int_{\Omega^*} f^*g^*$.
    \end{itemize}
\end{prop}
We refer to \cite{lieb2001analysis} chapter 3 for proofs of these facts.

\subsection{Isoperimetric inequality for Hopf symmetric sets and its consequences}
In this subsection we introduce the concept of Hopf symmetric set on $\KK\HH^n$ and the isoperimetric inequality. We then deduce the Faber-Krahn inequality and Chiti's comparison theorem from the isoperimetric inequality. We follow the paper by Lauro Silini \cite{silini2024approaching}.
\par Fix a point $p \in M$. Let $r= d(p,\cdot)$ be the distance function about $p$. Denote $\gamma(t)$ a unit-speed geodesic emanating from $p$ and  $\partial_r:= \gamma'(t)$ the unit radial vector field. Then for any $x\neq p$ and tangent vector $v\in T_xM$ not proportional to $\partial_r$, the sectional curvature of the two-subspace $Rm(v,\partial_r, \partial_r,v)$ is either -1 or -4. Denote $\mathcal{H}_x$ the subspace of those tangent vectors $v\in T_{x}M$ such that  $Rm(v,\partial_r, \partial_r,v)= -4$. The dimension of $\mathcal{H}_x$ is $k-1$ at each $x$. The union of those $\mathcal{H}_x$ defines a $k-1$ dimensional distribution. We say a set to be Hopf symmetric if:
\begin{definition}\label{DefHopfSymmetricSets}
   A $C^1$-set $E \subset M$ with normal vector field $\nu$ is said to be Hopf-symmetric if $\nu(x)$ is orthogonal to $\mathcal{H}_x$ at each point $x \in \partial E$, $p \notin \partial E$.
\end{definition}
\begin{remark}
    In $\HH^n$ the space $\mathcal{H}_x = \emptyset$, and all subsets $E \in \HH^n$ are Hopf symmetric.
\end{remark}

As remarked by Lauro, an example of Hopf-symmetric set comes from Hopf fibration:
\begin{example}
         Let $h: S^{kn-1} \rightarrow \mathbb{K} P^{n-1}$ be the Hopf fibration. Then, for any $C^1$ function $\rho: S^{kn-1} \rightarrow(0,+\infty)$ so that $\rho$ is constant along the fibers of $h$, the set with boundary

$$
\partial E:=\left\{\exp _p(\rho(x) x): x \in S^{kn-1} \subset T_p M\right\},
$$
is Hopf-symmetric, where $\exp _p$ is the exponential map of $M$ at an arbitrary point $p \in M$.
\end{example}

\par Recently Lauro Silini proved the following isoperimetric inequality for Hopf symmetric sets in non-compact rank one symmetric space\cite{silini2024approaching}:
\begin{theorem}
    In the class of Hopf-symmetric sets, geodesic balls are the unique isoperimetric sets in $\KK\HH^n$.
\end{theorem}

Using the isoperimetric inequality, we now have the Faber-Krahn inequality and Chiti's comparison theorem for Hopf-symmetric sets. The proof of Faber-Krahn is standard (see \cite[chapter~4]{chavel1984eigenvalues} and \cite[section~2.4]{ling2010bounds}); we just restrict the set to be Hopf-symmetric, since the isoperimetric inequality is only proven in that case.
\begin{theorem}[Faber–Krahn inequality for Hopf symmetric sets]\label{FaberKrahn}
     Let $M$ be one of the non-compact ROSS. Suppose $\Omega$ is a bounded Hopf-symmetric domain in $M$ with smooth boundary $\partial \Omega$. Suppose $\lambda_1$ is the first Dirichlet eigenvalue of the Laplace operator. Then
$$
\lambda_1(\Omega) \geq \lambda_1\left(\Omega^*\right) .
$$
where $\Omega^*$ is a geodesic ball whose volume is equal to that of $\Omega$.
Moreover, the equality holds if and only if $\Omega$ is a geodesic ball.
\end{theorem}

We now prove the following Chiti's comparison theorem for Hopf-symmetric sets. Recall $B_1= B_p(R)$ is the geodesic ball centered at $p\in M$ with radius $R$ such that $\lambda_1(B_1)= \lambda_1(\Omega)$. 
\begin{theorem}[Chiti's comparison theorem;See \cite{MR652376, MR652928,MR718816,benguria2007second}]\label{ChitiCompThm}
Let $M$ be one of the non-compact ROSS. Suppose $\Omega$ is a bounded Hopf-symmetric domain in $M$ with smooth boundary $\partial \Omega$. Let $u_1$ be the first Dirichlet eigenfunction of $-\Delta$ on $\Omega$, and let $z(r)$ be the first Dirichlet eigenfunction of $-\Delta$ on $B_1= B_p(R)$, normalized such that

$$
\int_{\Omega} u_1^2 =\int_{B_1} z^2  .
$$
Then there is some $r_0 \in(0, R)$ such that

$$
z(r) \geq u_1^*(r) \quad \text { for } r \in\left(0, r_0\right)
$$

and

$$
z(r) \leq u_1^*(r) \quad \text { for } r \in\left(r_0, R\right),
$$

\end{theorem}

\begin{proof}
    First, by Faber-Krahn inequality and domin monotonicity of Dirichlet eigenvalues, 
    \[\lambda_1(B_1)= \lambda_1(\Omega) \geq \lambda_1\left(\Omega^*\right) \Longrightarrow B_1 \subset\Omega^*.\]
Now the remaining of the proof is a minor variant of Lemma 9.1 in \cite{benguria2007second} if we change the volume of a geodesic ball from that of real hyperbolic space $A(r) = nC_n \int_0^r \sinh^{n-1}{t}dt$ to that of $\KK\HH^n$: $A(r) =nC_n \int_0^r \sinh^{kn-1}{t}\cosh^{k-1}{t}dt$.

\end{proof}
The following corollary from \cite[Corollary~3.3]{edelen2017ppw} is useful for our purpose:
\begin{cor}\label{ChitiCoro}
     If $F: \Omega^* \rightarrow \mathbb{R}_{+}$is a decreasing function of $r$, then

$$
\int_{\Omega^*}\left( u^*_1\right)^2 F \leq \int_{B_1} z^2 F
$$

If $F$ is an increasing function of $r$, then

$$
\int_{\Omega^*}\left(u^*_1\right)^2 F \geq \int_{B_1} z^2 F .
$$
with $z$ as in \cref{ChitiCompThm}.
\end{cor}


\subsection{Center of mass lemma}
We need the following center of mass lemma from \cite{Santhanam2007} to prove our results. To begin with, let $M$ be a complete Riemannian manifold. Let $q\in M$ be a point and $c(p)$ be the convexity radius at $p$. Note in the case of non-compact ROSS, the convexity radius is infinity. Now for a measurable set $E\subset B(q,c(q))$ for some point $q$, denote $CE$ to be its convex hull. Note $CE\subset B(q,c(q))$. Let $\exp _q: T_q M \to M$ be the exponential map and $X=\left(x_1, x_2, \ldots, x_n\right)$ be the geodesic normal coordinate centered at $q$. We identify $C E$ with $\exp _q^{-1}(C E)\subset T_q M$ and denote $g_q(X, X)$ as $\|X\|_q^2$ for $X \in T_q M$. The following lemma is called the center of mass lemma:
\begin{lemma}[Center of Mass Lemma]\label{CenterOfMassLemma}
    Let $E$ be a measurable subset of $M$ contained in $B\left(q, c\left(q\right)\right)$ for some point $q \in M$. Let $G:\left[0,2 c\left(q\right)\right] \rightarrow \mathbb{R}$ be a continuous function such that $G$ is positive on $\left(0,2 c\left(q\right)\right)$. Then there exists a point $p \in C E$ (called center of mass) such that

$$
\int_E G\left(\|X\|_p\right) X d V=0
$$

where $X=\left(x_1, x_2, \ldots, x_n\right)$ is a geodesic normal coordinate at $p$.
\end{lemma}
The proof is in \cite{Santhanam2007} and omitted here.

\section{The first two Dirichlet eigenfunctions of geodesic balls}\label{Section: First two Dirichlet of balls}
In this section, we use the method of separation of variables to identify the first two Dirichlet eigenfunctions of the Laplacian of a geodesic ball. Although our main theorem (\cref{MainTheorem}) is on noncompact spaces, we include both compact and noncompact cases here for completeness. 
\par Throughout this section, let $B$ be a geodesic ball of radius $R$ in $\KK\PP^n$ or $\KK\HH^n$. Let $u(r,\theta)= g(r)w(\theta)$ is an eigenfunction
\[\Delta_M u=-\lambda u \] where \[\Delta_M=\frac{\partial^2}{\partial r^2}+H(r) \frac{\partial}{\partial r}+\Delta_{S(r)},\] and   \[\text{Noncompact type: } H(r)=\frac{J'(r)}{J(r)}=(k n-1) \operatorname{coth} r+(k-1) \tanh r \]\[\text{Compact type: } H(r)=\frac{J'(r)}{J(r)}=(k n-1) \operatorname{cot} r-(k-1) \tan r \]is the mean curvature of the geodesic sphere of radius $r$.
Then 
\[-g^{\prime \prime}(r) \omega(\theta)-H(r) g^{\prime}(r) \omega(\theta)+\Delta_{S(r)} \omega(\theta)=\lambda  g(r) \omega(\theta)\]
Suppose $w(\theta)$ is an eigenfunction of $\Delta_{S(r)}$ with eigenvalue $\nu$. The equation now becomes
\[\begin{aligned}
& \left(-g^{\prime \prime}-H(r) g^{\prime}+ \nu g\right) \omega(\theta)=\lambda  g(r) \omega(\theta) \\
 \Longrightarrow & g^{\prime \prime}+H(r) g^{\prime}+(\lambda-\nu) g=0
\end{aligned}\]
The zero-th eigenvalue of $S(r)$ is $\nu_0= 0$ with eigenfunctions being constant, and the first is $\nu_1= \lambda_1(S(r))$ with eigenfunctions being the coordinate functions $X_i, i =1,2,...,n$. Thus the first eigenfunction $u_1(r,\theta) = g_1(r)$ must be radial and satisfies
\begin{equation}\label{FirstEigenODE}
    g_1^{\prime \prime}+H(r) g_1^{\prime}+\lambda_1(B)  g_1=0 \textit{ with } g_1'(0) = 0 \textit{ and } g_1(R) = 0.
\end{equation}
WLOG. we always assume the first eigenfunction to be positive.
\begin{prop}\label{g1decreasinglogconcave}
    The first Dirichlet eigenfunction of a geodesic ball is decreasing and log-concave.
\end{prop}
\begin{proof}
    We only prove the statement for balls in noncompact spaces here; for compact spaces, just replace the mean curvature function $H$ appropriately. 
    \par As $H(r) = \frac{J'(r)}{J(r)}$, multiplying $J(r)$ on both sides of \cref{FirstEigenODE} yieds
    \[\left( J g_1'\right)' = -\lambda_1 Jg_1<0.\] Hence $Jg_1'$ is decreasing in $r$. $Jg_1'\leq \lim_{r\to 0^+} J(r)g_1(r)= 0$, thus $g'_1 \leq 0$.
    \par For log-concavity, let $\varphi=\left(\log g_1\right)^{\prime}=\frac{g_1^{\prime}}{g_1}$ then $\varphi(0)=0$, $\varphi<0$ on $(0, R)$. We compute that
    \[\begin{aligned}
\varphi^{\prime} & =\frac{g_1^{\prime \prime}}{g_1}-\left(\frac{g_1^{\prime}}{g_1}\right)^2\\
&=-H(r)\left(\frac{g_1^{\prime}}{g_1}\right)-\lambda_1(B)-\left(\frac{g_1^{\prime}}{g_1}\right)^2 \\
& =-\left(\frac{kn-1}{\tanh r}+(k-1) \tanh r\right) \varphi-\lambda_1(B)-\varphi^2
\end{aligned}\] At $r=0:$
\[\begin{aligned}
     \varphi^{\prime}(0) &=-\lambda_1(B)-\lim _{r \rightarrow 0^{+}}\left(\frac{kn-1}{\tanh r}+(k-1) \tanh r\right) \varphi(r)\\
    &=-\lambda_1(B)-(k n-1) \varphi^{\prime}(0) .
\end{aligned}\]
Hence $\varphi'(0)<0.$ 
\begin{claim}
  $\varphi^{\prime}<0$ on $[0, R).$  
\end{claim}
Otherwise there is some $r_1\in (0,R)$ such that
\begin{equation*}
\left\{\begin{array}{l}
\varphi^{\prime}(r)<0 \text { on } (0,r_1)\\
\varphi^{\prime}\left(r_1\right)=0 . \\
\varphi^{\prime \prime}\left(r_1\right) \geqslant 0
\end{array}\right.
\end{equation*}
But \[\varphi^{\prime}=-H(r) \varphi-\lambda_1(B)-\varphi^2\] implies that
\[\begin{aligned}
\varphi^{\prime \prime} & =-H^{\prime}(r) \varphi-H(r) \varphi^{\prime}-2 \varphi \varphi^{\prime} \\
& =\lambda_1\left(S_r\right) \varphi-H(r) \varphi^{\prime}-2 \varphi \varphi^{\prime} .
\end{aligned}\]
At $r=r_1$: \[0 \leq \varphi^{\prime \prime}\left(r_1\right)=\lambda_1\left(S_{r_1}\right) \varphi\left(r_1\right)<0.\] 
A contradiction. This completes the proof of the claim and the lemma.

\end{proof}

We now turn to the second eigenfunction $u_2(r,\theta)= g_2(r)w_i(\theta)$. The radial part $g_2(r)$ must satisfy either
\begin{equation}\label{EqnLdmda02}
    f^{\prime \prime}+H(r) f^{\prime}+\lambda_{0,2} f=0
\end{equation}
or 
\begin{equation}\label{EqnLamda11}
    g^{\prime \prime}+H(r) g^{\prime}+\left(\lambda_{1,1}\left(B\right)-\lambda_1(S_r)\right) g=0.
\end{equation}
 We now show that the first case is impossible and $\lambda_2(B) = \lambda_{1,1}$.
\begin{prop}
    The radial part of the second eigenfunction satisfies
    \begin{equation}\label{SecondEigenODE}
        g_2^{\prime \prime}+H(r) g_2^{\prime}+\left(\lambda_{2}\left(B\right)-\lambda_1(S_r)\right) g_2=0 \textit{ with } g_2'(0) = 0 \textit{ and } g_2(R) = 0.
    \end{equation}
\end{prop}
\begin{proof}
   Firstly note that $f$ above is the second eigenfunction to system \cref{EqnLdmda02} and therefore changes sign once on $(0,R)$; while $g$ is the first eigenfunction to the system \cref{EqnLamda11} and therefore does not change sign. WLOG we choose $g$ to be positive. Let $h$ be a non-trivial solution of \[h^{\prime \prime}+H(r) h^{\prime}+\lambda_{1,1} h=0.\] Then by differentiation we see that $h'$ satisfies \cref{EqnLamda11} for $g$ above with $\lambda= \lambda_{1,1}$. That implies $h'$ is proportional to $g$ and WLOG we assume $h' =g >0$. Then $h$ is increasing on $(0,R)$. Since $f$ and $h$ both satisfies the same equation with eigenvalues $\lambda_{0,2}$ and $\lambda_{1,1}$ respectively, they are both first eigenfunctions to some smaller balls of radius $R_f $ and $ R_h$. Assume by contradiction that $\lambda_{0,2}<\lambda_{1,1}$. By Pr\"{u}fer transformation\cite[Chapter~10]{BirkhoffODE1989}, the first positive zero of first eigenfunction to \cref{EqnLdmda02}is decreasing in $\lambda$. Hence the first positive zero of $h$, $R_h$ appears first, then $R_f$, i.e. $R_h < R_f$. Since $h$ is increasing, it must happen that $h<0$ on $(0, R_h)$, $h>0$ on $(R_h, R)$. We suppose $f>0$ on $(0,R_f)$. Now subtract the equations of $f$ and $h$ and integrate:
   \begin{align*}
       \int_{R_f}^R\left(\lambda_{0,2}-\lambda_{1,1}\right) f h J &=\left.J(r)\left(f h^{\prime}-h f^{\prime}\right)\right|_{R_f} ^R \\
       &= J(R)(f(R)h'(R)-h(R)f'(R)) - J(R_f)(f(R_f)h'(R_f)- h(R_f)f'(R_f))
   \end{align*}
   Observe that $f(R_f) = 0, f(R) = 0, h(R)\geq 0$ and $h(R_f)>0, f'(R_f)<0$. We see that LHS is positive while RHS is negative. Contradiction! This completes the proof of this lemma.
\end{proof}

From \cref{FirstEigenODE} and \cref{SecondEigenODE} we are able to derive the local behavior of eigenfunctions by the method of Frobenius-Taylor expansion.
\begin{lemma}\label{EigenfunctionExpansion}
    The radial part of eigenfunctions $g_i$ satisfies the following asymptotic expansion at $r=0$: for noncompact ROSS
    \begin{equation}
\left\{\begin{array}{l}
g_1=1-\frac{\lambda_1}{2 k n} r^2+o\left(r^2\right) \\
g_2=r-\frac{\frac{2}{3} k n+2 k-\frac{8}{3}+\lambda_2}{2 k n+4} r^3+o\left(r^3\right)
\end{array}\right.
\end{equation}
and for compact ROSS    
\begin{equation}
\left\{\begin{array}{l}
g_1=1-\frac{\lambda_1}{2 k n} r^2+o\left(r^2\right) \\
g_2=r+\frac{\frac{2}{3} k n+2 k-\frac{8}{3}-\lambda_2}{2 k n+4} r^3+o\left(r^3\right)
\end{array}\right.
\end{equation}
At $r=R$ for either compact or noncompact ROSS:\[g_1=(r-R)-\frac{H(R)}{2}(r-R)^2+\frac{1}{6}\left(H(R)^2-\lambda_1\right)(r-R)^3+\cdots\] and
\[g_2=r-R-\frac{H(R)}{2}(r-R)^2+\frac{1}{6}\left(H^2(R)-\lambda_2-H^{\prime}(R)\right)(r-R)^3+\cdots\]
\end{lemma}

\begin{proof}
    For both ODEs $r=0$ is a regular singular point. Assume $g_i(r)=r^i \sum_{n=0}^{\infty} a_n r^n$ and then substitute into the corresponding ODE. For $g_1$, it is sufficient to consider the ODE with the first order expansion of $H(r)= \frac{kn-1}{r}+(k-1+\frac{kn-1}{3})r+...$ at $r=0$: \[g_1^{\prime \prime}+(\frac{kn-1}{r}+(k-1+\frac{kn-1}{3})r+...)g_1^{\prime}+\lambda_1 g_1=0.\]
    Now suppose $g_1(r)=r^m \sum_{n=0}^{\infty} a_n r^n$, substituting into the ODE gives (up to order $r^2$)
    \[2a_2+6a_3r+(\frac{kn-1}{r}+(k-1+\frac{kn-1}{3})r)(a_1+2a_2r)+(k-1+\frac{kn-1}{3})a_1r+\lambda_1+\lambda_1 a_1r=0.\]
    Comparing the coefficients for each power of $r$, the $n=0$ term gives \[a_0 m(m-1)+(k n-1) a_0 m=0,\] and so \[m_1=0, \quad m_2=2-k n<0.\] The reasonable solution is $m=0$. Then $g_1(r)=\sum_{n=0}^{\infty} a_n r^n$. We can take $a_0=1$ for simplicity. Then the $n=1$ term gives $a_1=0$, and the $n=2$ term gives $a_2=\frac{-\lambda_1}{2 k n}$. Hence we get \[g_1=1-\frac{\lambda_1}{2 k n} r^2+o\left(r^2\right)\] as desired. The other cases are similar.
\end{proof}

We now derive a useful gap estimate $\lambda_2-\lambda_1$ on geodesic balls:

\begin{prop}\label{G''(R) negative}
       Let $R>0$ and $B_R$ be a geodesic ball of radius $R$ in $\KK\PP^n$ or $\KK\HH^n$. Assume $R\leq \pi/4$ in $\KK\PP^n$ case. Then $\lambda_2\left(B_R\right)-\lambda_1\left(B_R\right) \geqslant \lambda_1\left(S_R\right).$ 
\end{prop}
    \begin{remark}
         In \cite{ashbaugh2001sharp} Ashbaugh and Benguria showed this inequality in their Lemma 3.4 in the case of balls in hemisphere; in \cite[Section 2]{wei2022fundamental} Nguyen, Stancu and Wei showed this inequality for balls in $\HH^n$. Now we generalize it to ROSS. 
    \end{remark}
\begin{proof}
    We first prove the statement for $\KK\HH^n$. By a transformation $g_1=(\sinh r)^{\frac{1-k n}{2}}(\cosh r)^{\frac{1-k}{2}} \bar{u}_1$ the ODE for $g_1$ becomes
    \[-\bar{u}_1^{\prime \prime}+V_1 \bar{u}_1=\lambda_1 \bar{u}_1 .\] where the potential $V_1$ is a function of $r, H(r), k, n$:
  \begin{align*}
      V_1 =&-\frac{1}{4}[-1+k^2+2\left(3+k^2 n-2 k(n+1)\right)\coth^2{r} +\left(1-k^2 n^2\right) \operatorname{coth}^4 r \\
           &-2 \operatorname{coth} r\left(-1+k+(k n-1) \operatorname{coth}^2 r\right) H(r)] \tanh ^2 r.
  \end{align*}
For $g_2$, the same substitution  $g_2=(\sinh r)^{\frac{1-k n}{2}}(\cosh r)^{\frac{1-k}{2}} \bar{u}_2$ yields
\[-\bar{u}_2^{\prime \prime}+V_2 \bar{u}_2=\lambda_2 \bar{u}_2 .\] where 
\[V_2 :=V_1+\lambda_1(S_r)\geq V_1 + \lambda_1(S_R)\] since $\lambda_1(S_r)$ is decreasing in $r$. We see that $\bar{u_i}$ is the first eigenfunction of the one-dimensional Schrödinger operator $L_i= -\frac{d^2}{dr^2}+V_i$ with eigenvalue $\lambda_i$.
Therefore, we get\[\lambda_2-\lambda_1 \geqslant \lambda_1\left(S_R\right)\] as desired. 
\par For  $\KK\PP^n$, observe that in the above arguments, all we need is the fact that $\lambda_1(S_r)$ is decreasing. Since this is true in $(0,\pi/4]$, mimicking the arguments above yeids the desired inequality.

\end{proof}

\section{Proof of the main theorem}\label{section: Main thoerem proof}
We now turn to the proof of the main theorems. We follow the structure as in \cite[Section 5]{benguria2007second}. Fix a bounded Hopf-symmetric domain $\Omega$ in $\KK\HH^n$. Let $B_1= B(p, R)\subset \KK\HH^n$ be a geodesic ball such that $\lambda_1(B_1)=\lambda_1(\Omega)$. Denote $u_i$ be the i-th Dirichlet eigenfunction of $\Omega$, and $g_i$ be the radial part of the i-th Dirichlet eigenfunction of $B_1$. A variant of Rayleigh characterization (see, for example, \cite[Section 2.4]{miker2009eigenvalue}) says
\[\lambda_2(\Omega)-\lambda_1(\Omega) = \inf \frac{\int_{\Omega}|\nabla P|^2 u_1^2 }{\int_{\Omega} P^2 u_1^2} \]
 where the infimum is taken among all functions $P \in H^1(\Omega)$ satisfying
 \[\int_{\Omega} P u_1^2 =0.\]
To find an upper bound, we throw in a suitable test function $P$. We choose the following functions to be our test function:
\[P_i= G(r)\frac{X_i}{r}\] where $(X_1,\cdots,X_n)$ is a geodesic normal coordinate centered at some point $p$ to be chosen later and $r:= d(\cdot,p)$ is the radial distance to $p$. The radial part $G(r)$ is defined as 
\[G(r)= \begin{cases}\frac{g_2(r)}{g_1(r)} & \text { for } r \in(0, R), \\ \lim _{t \uparrow R} G(t) & \text { for } r \geq R .\end{cases}\]
By the center of mass lemma (\cref{CenterOfMassLemma}) there is $p$ such that $P_i$ satisfies $\int_{\Omega} P_i u_1^2 =0$ for all $i$. We compute 
\[\sum_{i=1}^n P_i^2=G^2(r),\] and 
\[\sum_{i=1}^n|\nabla P_i|^2 = G'(r)^2+\lambda_1(S_r)G^2.\] Thus by the Rayleigh characterization above, we get
\begin{equation}\label{lamda_2-lamda_1}
                \lambda_2-\lambda_1 \leq \frac{\int_{\Omega}\left[ G'(r)^2+\lambda_1(S_r)G^2  \right] u_1^2 d x}{\int_{\Omega} G^2 u_1^2 d x}.
\end{equation}
Define the numerator in \cref{lamda_2-lamda_1}\[B := G'(r)^2+\lambda_1(S_r)G^2. \]
We now need the monotonicity of $B$ and $G$ (see \cref{Monotonicity Lemma}). We use the properties of spherical rearrangement 
\[\begin{aligned}
            \int_{\Omega} u_1^2 B(r) d V & \leq \int_{\Omega^*} u_1^*(r)^2 B^*(r) d V \\
            & = \int_{\Omega^*} u_1^*(r)^2 B(r) d V
            \end{aligned}\] 
            and\[\begin{aligned}
            \int_{\Omega} u_1^2G(r)^2 d V & \geq \int_{\Omega^*} u_1^*(r)^2 G_*(r)^2 d V \\
            & = \int_{\Omega^*} u_1^*(r)^2 G(r)^2 d V
            \end{aligned}\]
To move from $\Omega^*$ to $B_{1}$, we apply \cref{ChitiCoro} and monotonicity of $G$ and $B$ to get
\begin{align*}
                 \int_{\Omega} u_1^2 B(r) d V &\leq \int_{\Omega^*} u_1^*(r)^2 B(r) d V\leq \int_{B_{1}} z_1^2(r) B(r) d V \\
                 \int_{\Omega} u_1^2 G(r)^2 d V &\geq \int_{\Omega^*} u_1^*(r)^2 G(r)^2 d V \geq \int_{B_{1}} z_1^2(r) G(r)^2 d V
\end{align*}
Apply again the Rayleigh characterization,
\[\lambda_2(\Omega)-\lambda_1(\Omega) \leq \frac{\int_{B_1} z_1^2(r) B(r) \mathrm{d} V}{\int_{B_1} z_1^2(r) G^2(r) \mathrm{d} V}=\lambda_2\left(B_1\right)-\lambda_1\left(B_1\right) .\]
As $\lambda_1\left(B_1\right)= \lambda_1\left(\Omega\right)$, this completes the proof of the main theorem.

\section{Monotonicity of the functions $B$ and $G$}\label{Section: MonotonicityLemma}
Recall that \[B = G'(r)^2+\lambda_1(S_r)G^2 \textit{ and }G =\frac{g_2(r)}{g_1(r)}\text { for } r \in(0, R). \]
Our goal in this section is to show the following:
\begin{prop}[Monotonicity Proposition]\label{Monotonicity Lemma}
    The functions $B$ is decreasing in $r$ in noncompact ROSS and $G$ is increasing in $r$ in either compact (for $0<r\leq \pi/4$) or noncompact ROSS. 
\end{prop}

We first show that $G$ is increasing:
\begin{lemma}\label{LemmaGisIncreasing}
     Let $R>0$ and $B_R$ be a geodesic ball of radius $R$ in $\KK\PP^n$ or $\KK\HH^n$. Assume $R\leq \pi/4$ in $\KK\PP^n$ case. Then the quotient of the radial part of the first two eigenfunctions of geodesic balls $G$ is increasing.
\end{lemma}
\begin{proof}
    First $ G^{\prime}=G\left(\frac{g_2'}{g_2}-\frac{g_1^{\prime}}{g_1}\right)$. We will show $G'>0$.
    We choose $g_1$ to be positive, symmetric on $[-R, R]$ and $g_2$ to be odd. Assume on the contrary $G$ is not always increasing, i.e. $G'$ is negative somewhere. We compute
    \[G^{\prime \prime}=\left[\frac{g_2^{\prime \prime}}{g_2}-\left(\frac{g_2^{\prime}}{g_2}\right)^2-\frac{g_1^{\prime \prime}}{g_1}+\left(\frac{g_1^{\prime}}{g_1}\right)^2\right] G+\left(\frac{g_1}{g_2}-\frac{g_1^{\prime}}{g_1}\right) G^{\prime}\]
    Eliminating $g_i''$ with their ODEs, we get
    \[G''=G\left(\lambda_1(B)-\lambda_2(B)+\lambda_1\left(S_r\right)\right)-H(r) G'+\frac{G^{\prime} \cdot G^{\prime}}{G}+\left(\frac{g_1^{\prime}}{g_1}+\frac{g_2^{\prime}}{g_2}\right)G'\]
    At $r$ where $G'(r)=0$:
    \begin{equation}\label{G''atG'=0}
        \left.G^{\prime \prime}\right|_{G^{\prime}=0}=G\left(\lambda_1(B)-\lambda_2(B)+\lambda_1\left(S_r\right)\right)
    \end{equation}
    By \cref{EigenfunctionExpansion} we get the asymptotic expansion of $G$:
    \[G=\frac{g_2}{g_1} \sim \frac{r}{1}=r \text { as } r \rightarrow 0^{+}\] and
    \[G(0)=\frac{g_2(0)}{g_1(0)}=0, \quad G^{\prime}(0)>0.\]
   By the expansion at $r=R$ we get $G'(R)= 0$.  \cref{G''(R) negative} and \cref{G''atG'=0} implies $G''(R)<0$. We know that $G'>0$ on some interval $(R-\epsilon,R)$. There are $\alpha, \beta \in (0,R-\epsilon)$ such that 
   \[\left\{\begin{array}{l}
G^{\prime}(\alpha)=G^{\prime}(\beta)=0 \\
G^{\prime \prime}(\alpha) \leq 0 . \quad G^{\prime \prime}(\beta) \geqslant 0 .
\end{array}\right.\]
    
    Since both $G$ and $\lambda_1(S_r)$ are decreasing functions of $r$ on the interval $(\alpha,\beta)$, RHS of the above expression is decreasing in $r$ and hence $G''(\alpha)\geq G''(\beta)$. But this contradicts with $G^{\prime \prime}(\alpha) \leq 0, \text{ and } G^{\prime \prime}(\beta) \geqslant 0$.
\end{proof}

So far we have shown $G'\geq0$. Our next target is to show $B'\leq0$ in $\KK\HH^n$. The following arguments are rather technical, and we mostly follow the corresponding section in \cite{benguria2007second}. We first differentiate $B$ and use \cref{Lambda1Sr} to get:
\begin{equation}\label{Bprime}
    B'=2G'G''+ 2G^2\psi(r)
\end{equation}
where \begin{equation}\label{SecondTermBprime}
    \psi(r):=\left(\frac{k n-1}{\sinh^2{r}}-\frac{k-1}{\cosh^2 r}\right) \frac{G'}{G}-\left(\frac{k n-1}{\sinh^2 r} \coth{r}-\frac{k-1}{\cosh^2 r} \tanh{r} \right).
\end{equation}
If we could show $G''\leq0$ and $\psi\leq0$ then we are done. The condition $\psi\leq 0$ is equivalent to 
\begin{equation}\label{RequirementsOfG'/G}
    \frac{G'}{G} \leqslant \frac{(k n-1) \operatorname{coth}r \cdot \cosh^2 r-(k-1) \tanh r \cdot \sinh^2 r}{(k n-1) \cosh ^2 r-(k-1) \sinh^2 r}.
\end{equation}
The right-hand side of the above inequality is greater than or equal to $\coth r$, which is further greater than $1/r$. Hence our goal reduces to show 
\begin{equation}\label{DefinitionOfQ}
    q(r):= r\frac{G'(r)}{G(r)}\leq 1
\end{equation}
on $(0,R)$, for any fixed radius $R$. 
\begin{remark}\label{remark:CompactRossQ}
    In the case of \textbf{compact} ROSS, showing $G'/G\leq 1/r$ is not sufficient, as \cref{RequirementsOfG'/G} becomes 
    \begin{equation}
        \frac{G^{\prime}}{G} \leq \frac{(k n-1) \cot{r} \cdot \cos^2{r}-(k-1) \sin^2 {r} \tan r}{(k n-1) \cos^2{r}+(k-1) \sin^2 {r}}< \frac{1}{r}.
    \end{equation}
    One has to deal with the complicated right-hand side directly.
\end{remark}
\par As in \cite{benguria2007second} Section 7, to prove \cref{Monotonicity Lemma} it is enough to show 
\begin{equation}\label{RequirementsForQ}
\begin{aligned}
& q(r) \geq 0, \\
& q(r) \leq 1,\\
& q^{\prime}(r) \leq 0
\end{aligned}
\end{equation}
on $(0,R)$.
This implies $G''\leq 0$ as well, as we compute 
\[G^{\prime \prime}=\frac{G}{r^2}\left(r q^{\prime}+q(q-1)\right) .\]
The rest of the section is devoted to prove \cref{RequirementsForQ}.
\par By \cref{EigenfunctionExpansion} and \cref{DefinitionOfQ} we can deduce the boundary asymptotic behavior of $q$:
\begin{equation}
\lim _{r \rightarrow 0} q(r)=1 \quad \text { and } \quad \lim _{r \rightarrow R} q(r)=0.
\end{equation}
We now differentiate $q$:
\begin{equation}\label{DerivativeOfQ}
q'=\frac{q(1-q)}{r}-H(r) q+\lambda_1\left(S_r\right) \cdot r+\left(\lambda_1-\lambda_2\right) r-2 p q
\end{equation} where $p= g_1'/g_1$ satisfies
\begin{equation}\label{EquationOfPprime}
p^{\prime}=-p^2-H(r) p-\lambda_1
\end{equation}
As in \cite{benguria2007second} we set
\begin{equation}\label{T r y}
T(r,y):=\frac{y(1-y)}{r}-H(r) y+\lambda_1\left(S_r\right) \cdot r+\left(\lambda_1-\lambda_2\right) r-2 p(r) y \quad \text{for}\quad 0<y<1
\end{equation} to be the directional field of $q$ in the sense that $q'(r)= T(r,q(r)).$
We compute derivative of $T$ with respect to $r$-variable:
\begin{equation}
T^{\prime}:=\frac{\partial}{\partial r} T(r, y)=\frac{y^2-y}{r^2}+\lambda_1\left(S_r\right)(y+1)+\left(\lambda_1\left(S_r\right)\right)^{\prime} \cdot r-2 p^{\prime}(r) \cdot y-\left(\lambda_2-\lambda_1\right).
\end{equation}
We are interested in the behavior of $T'$ when $T= 0$. To do that we first compute the value of $p$ when $T=0$ via \cref{EquationOfPprime}:
\begin{equation}
p_{T=0}=\frac{1}{2 y}\left(\frac{y(1-y)}{r}-H(r) y+\lambda_1\left(S_r\right) \cdot r-\left(\lambda_2-\lambda_1\right) r\right).
\end{equation}
Substituting back to \cref{EquationOfPprime} we get $\left.p^{\prime}\right|_{T=0}=-\left.\left.p^2\right|_{T=0} H(r) p\right|_{T=0}-\lambda_1$ and finally we reach
\begin{equation}\label{T'WhenT=0}
\begin{aligned}
\left.T^{\prime}\right|_{T=0} & =\frac{y^2-y}{r^2}+\lambda_1\left(S_r\right)(y+1)-\left(\lambda_2-\lambda_1\right)+\left(\lambda_1\left(S_r\right)\right)^{\prime} \cdot r-2 p^{\prime}(r)|_{T=0} \cdot y \\
& =\frac{y^2-y}{r^2}+\lambda_1\left(S_r\right) \cdot(y+1)+\left(\lambda_1\left(S_r\right)\right)^{\prime} \cdot r+2 \lambda_1 y-\left(\lambda_2-\lambda_1\right) \\
& +\frac{1}{2 y}\left(\frac{y(1-y)}{r}-H(r) y+\lambda_1\left(S_r\right) \cdot r+\left(\lambda_1-\lambda_2\right) r\right)^2 \\
& +H(r) \quad\left(\frac{y(1-y)}{r}-H(r) y+\lambda_1\left(S_r\right) \cdot r+\left(\lambda_1-\lambda_2\right) r\right).
\end{aligned}
\end{equation}
We call the right-hand side of \cref{T'WhenT=0} as $Z_y(r)$:
\[Z_y(r):=\text{R.H.S of \cref{T'WhenT=0}.} \]
The corresponding asymptotic behaviors of $p, T'$ and $Z_1(r)$ are similar to \cite[page 267]{benguria2007second} Section 7:
\begin{itemize}
    \item $\lim _{r \rightarrow 0^{+}} p^{\prime}(r)=-\frac{\lambda_1}{k n}$,
    \item $\lim _{r \rightarrow 0^{+}} T(r, 1)=0$,
    \item $\lim _{r \rightarrow 0^{+}} T^{\prime}(r, 1)=-\lambda_2+\left(1+\frac{2 \lambda_1}{k n}\right) \lambda_1+\frac{2}{3}(-k n-3 k+4)$,
    \item $\lim _{r \rightarrow 0^{+}} z_1(r)=k n\left(-\lambda_2+\left(1+\frac{2}{k n}\right) \lambda_1+\frac{2}{3}(4-k n-3 k)\right)$.
\end{itemize}
Parallel to \cite{benguria2007second} Lemma 7.2, it is sufficient to show the following.
\begin{lemma}[Properties of $Z_y$]\label{PropertiesOfZy}
    We have the following: 
    \begin{enumerate}[a.]
        \item There is no pair $r$, $y$ with $r \in(0, R)$ and $0<y<1$ such that $Z_y^{\prime}(r)=0$ and $Z_y^{\prime \prime}(r) \leq 0$.\\
        \item The function $Z_1(r)$ is strictly increasing in $r$ on the interval $(0, R)$.
    \end{enumerate}

\end{lemma}

\par We  now prove \cref{PropertiesOfZy}. We follow the proof in \cite{benguria2007second} Section 8.


\par We decompose the function $Z_y$ into several parts:
\begin{equation}
    \begin{aligned}
        Z_y& =\frac{y}{2}((kn-1)^2 A_1+(k-1)^2 B_1)+\frac{1}{2}\left(y-y^3\right)A_2+\frac{8(k-1)(k n-1)}{2y}B_2 \\
        &+\frac{1}{2y}((kn-1)^2 A_3+(k-1)^2 B_3)+\frac{(\lambda_2-\lambda_1)}{y}((kn-1)A_4+(k-1)B_4)\\
        &+2((kn-1)A_5+(k-1)B_5)+\frac{(\lambda_2-\lambda_1)^2}{2y}A_6+ C 
    \end{aligned}
\end{equation}
with
\begin{equation}
\begin{array}{ll}
A_1=-\operatorname{coth}^2 r, & A_2=-r^{-2} \\
A_3= r^2 \sinh ^{-4} r, & A_4=-r^2 \sinh ^{-2} r \\
A_5=-(r \operatorname{coth} r-1) \sinh ^{-2} r, & A_6= r^2\\
B_1=-\tanh^2{r}, & B_2=-r^2 \sinh ^2(2 r) \\
B_3=  r^2\cosh ^{-4} r,& B_4=r^2\cosh^{-2}{r}\\
B_5= (r \tanh r-1) \cosh ^{-2} r, & C= \text{all terms independent of } r
\end{array}
\end{equation}
The grouping rules are as follows:
\begin{itemize}
    \item All $A_i$ terms are consistent with that in \cite{benguria2007second} Section 8.
    \item All $B_i$ terms are new in our case. If taking $k=1$, $B_i= 0$ and one reduces to $\HH^n$ case as in \cite{benguria2007second}.
    \item All coefficients are positive. Note $0<y<1$.
\end{itemize}

\subsection{Proof of \cref{PropertiesOfZy}(a)}
We have to show that for all $r>0$ and $0<y<1$, the case $Z_y^{\prime}(r)=0 \text { and } Z_y^{\prime \prime}(r) \leq 0$ never happens. The trick to regard $\vec{Z}= (Z', Z'')$ and similarly $\vec{A_i}, \vec{B_i}$ as vectors on the XOY-plane $\RR^2$. If we can show that $\vec{Z}$ never lies on the negative part of Y-axis, then we are done. We claim that we can pick a reference vector $\vec V$ from $\vec A_i, \vec B_i$ such that all others lie in the same one of the two half spaces with boundary $\{c \vec V| c\in \RR\}$ while the negative Y-axis lies in the other half space. Then all positive linear combinations of $\vec A_i, \vec B_i$ cannot touch the negative Y-axis. We define the cross product of vectors $\vec A= (A',A'') \text{  and  } \vec B$ as
\[\vec A\times \vec B = A'B''-A''B'.\]
In \cite{benguria2007second} Benguria and Linde choose $\vec A_1$ as the reference vector $\vec{V}$ and showed that
\[\vec A_1\times \vec A_i>0\] for all $i=2,3,...,6$. In our case we choose the reference vector $\vec{V}$ to be $(kn-1)^2\vec A_1+(k-1)^2\vec B_1$ for $r\leq r_0$ ($r_0$ to be determined in the next lemma) and $\vec B_2$ for $r>r_0$. More specifically, we prove
\begin{lemma}\label{Crossproduct positivity lemma}
    There is some $r_0= r_0(k,n)>0$ depending on $k,n$ such that for all $0<r<r_0$
    \[((kn-1)^2\vec A_1+(k-1)^2\vec B_1)\times ((kn-1)\vec A_i+(k-1)\vec B_i) >0, \text{ for } i =4,5\]
    \[((kn-1)^2\vec A_1+(k-1)^2\vec B_1)\times ((kn-1)^2\vec A_3+(k-1)^2\vec B_3) >0\]
    and 
    \[((kn-1)^2\vec A_1+(k-1)^2\vec B_1)\times \vec A_2>0, ( (kn-1)^2 \vec A_1+(k-1)^2 \vec B_1)\times \vec B_2 >0.\]
    For $r>r_0$ we have
    \[\vec B_2 \times ((kn-1)^2\vec A_i+(k-1)^2\vec B_i)>0, \text{ for } i =1,3\]
    \[\vec B_2 \times ((kn-1)\vec A_i+(k-1)\vec B_i)>0, \text{ for } i =4,5\]
    and \[\vec B_2 \times \vec A_2>0.\]
\end{lemma}
The proof of this lemma is purely calculus and is attached in the appendix at the end of this paper for interested readers.

\subsection{Proof of \cref{PropertiesOfZy}(b) }
In this subsection we now show that $Z_1$ is increasing in $r$. We calculate
\begin{align*}
    Z_1 &=\frac{1}{2}((kn-1)^2 A_1+(k-1)^2 B_1)+\frac{8(k-1)(k n-1)}{2}B_2 \\
        &+\frac{1}{2}((kn-1)^2 A_3+(k-1)^2 B_3)+\frac{(\lambda_2-\lambda_1)}{1}((kn-1)A_4+(k-1)B_4)\\
        &+2((kn-1)A_5+(k-1)B_5)+\frac{(\lambda_2-\lambda_1)^2}{1}A_6+ C 
\end{align*}
Now we show each term above is increasing in $r$.
\par (1) $A_6= r^2$ is increasing in $r$.
\par (2) $((kn-1)A_5+(k-1)B_5)'=\frac{1}{4 \sinh^4 r \cosh^4 r}(8 r+4 r \cosh (4 r)-3 \sinh (4 r))$.
Call the numerator $f(r):=8 r+4 r \cosh (4 r)-3 \sinh (4 r)$. Then $f(0)=\cdots=f''(0)=0$, and $f^{(3)}(r)=266 r \sinh (4 r) \geqslant 0$, therefore $f\geq 0$.
\par (3) $(kn-1)A_4+(k-1)B_4 = r^2\left(-\lambda_1(S_r)\right)$. Since $\lambda_1(S_r)$ is decreasing, we see the term is increasing.
\par (4) For $i=1,3$, we combine the two terms together as $(k n-1)^2\left(A_1+A_3\right)+(k-1)^2\left(B_1+B_3\right)$. Benguria and Linde \cite{benguria2007second} showed that $A_1+A_3$ is increasing, and one can check that $B_1+B_3$ is decreasing. To show that the whole term is increasing overall, we use induction in $n$. The base case $n=2$ is given by
\[(2k-1)^2 (A_1+A_3)+ (k-1)^2 (B_1+B_3).\] We only prove the case $k=2$ for $\CC\HH^n$ and the other cases are similar. When $k=2$, \[f(r)= 9(A_1+A_3)+  (B_1+B_3).\] To show that $f$ is increasing, we compute its derivative 
\[f' = \frac{18\cosh{r}}{\sinh^5{r}}(r \tanh r+\sinh ^2 r-2 r^2)+\frac{2\sinh{r}}{\cosh^5{r}}(r \coth r-\cosh^2 r-2 r^2)\]
Clearing out the denominators, we set
\[ g(r) = f'(r)\sinh^4{r}\cosh^4{r}= \cosh^4{r}X(r)+\sinh^4{r} Y(r), r>0\]
where $X(r)= 18\cosh{r}\sinh r+18r-36r^2\coth r$ and $Y(r)= 2r-2\cosh{r}\sinh r-4r^2 \tanh r$. We can check that $X(r)>0$ and $Y(r)<0$. As $\cosh r>\sinh r$ for all $r>0$, $g>0$ if $X= |X|> \abs{Y}= -Y \Longleftrightarrow X+Y>0$. 
\begin{itemize}
    \item For $r\geq 2$, $e^{-r} \leq e^{-2}< 0.136$ and $\coth{r}\leq 1.05$, so 
    \begin{align*}
        X+Y &= 20r+8\sinh{2r}-36r^2 \coth r-4r^2 \tanh r \\
        &\geq 4(e^{2r}-e^{-4})+20r-36\times 1.05 r^2-4r^2>0
    \end{align*}
    \item For $0<r<1$, Taylor expansion of $g$ at $r=0$ gives $g\sim 4 x+\frac{28 x^3}{3}+\frac{112 x^3}{9}+\dots$, so $g\geq 4x$ and $g>0$.
    \item For intermediate values $1<r<2$, we attach the graph of $g$ (\cref{Graph of $g$ Lemma 5.3B}) showing that it is indeed positive there.
    \begin{figure}
    \centering
    \includegraphics[width=0.5\linewidth]{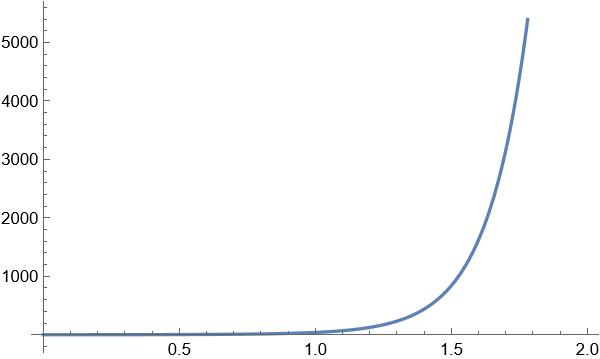}
    \caption{Graph of $g$ Lemma 5.3B}
    \label{Graph of $g$ Lemma 5.3B}
\end{figure}
\end{itemize}

\par Increasing $n$ is equivalent to adding more $A_1+A_3$ terms.

\par (5) Lastly, $B_2 =-\frac{1}{r^2 \sinh r^2}$ is clearly increasing in $r$.
\par This completes the proof of \cref{PropertiesOfZy}. $\hfill \square$
\newline
\par Now with \cref{PropertiesOfZy}, we are ready to prove \cref{RequirementsForQ} as in \cite[page 268-270]{benguria2007second}. For completeness of the paper, we present the proof here. The proof is in several steps.

\begin{claim}
    $q$ is nonnegative on $(0,R)$.
\end{claim}
\begin{proof}
    \cref{LemmaGisIncreasing} shows $G'\geq 0$ and by definition $q= r\frac{G'}{G}$.
\end{proof} 

\begin{claim}
    $q<1$ on a small neighborhood $(0,\epsilon)$ of $0$.
\end{claim}
\begin{proof}
    The asymptotic behavior of $q$ is that $q \rightarrow 1$ as $r \rightarrow 0^+$. Since $T'(\cdot,1)= Z_1(\cdot)$ and $Z_1$ is strictly increasing by \cref{PropertiesOfZy}(b), $T(r,1)$ has at most two zeros in $(0,R)$.
    \par Case (i) $T(r,1)<0$ on $(0,\epsilon)$. Then by \cref{T r y} we get \[\frac{\partial T}{\partial y}|_{y=1}=-\frac{1}{r}-H(r)-2p(r). \]
    Since $p \rightarrow 1$ as $r \rightarrow 0^+$, $\frac{\partial T}{\partial y}|_{y=1}<0$ on $(0,\epsilon)$ for $\epsilon$ small enough. Fixing $r$, the function $T(r,y)$ is a parabola in $y$. So $T<0$ on $(0, \epsilon) \times[1, \infty)$, which implies $q<1$ on $(0,\epsilon)$ as $q'= T(r,q(r))$.
    \par Case (ii) $T(r,1)>0$ on $(0,\epsilon)$. Since $T(r,1)\rightarrow 0$ as $r\rightarrow 0^+$, $T'(r,1)\geq 0$ on $(0,\epsilon)$. Hence $Z_1(0)\geq 0$. As $Z_1$ increases strictly, so $Z_1>0$ on $(0,\epsilon)$. This further implies $T(r,1)>0$ on $(0,R)$. If $q\geq 1$ somewhere, then, as $\lim _{r \rightarrow R} q(r)=0$, there must be some $r_1 \in (0,R)$ such that $q(r_1)=1$ and $q'(r_1)= T(r_1,1)\leq 0$. A Contradiction to $T(r,1)>0$.
\end{proof}

\begin{claim}
    If $0 <q(r)<1$, then $q'(r) \leq 0$.
\end{claim}
\begin{proof}
    If not, there were 3 points $r_1,r_2,r_3$ such that $0<q_0:=q(r_1)=q(r_2)=q(r_3)<1$ and $q'(r_1)<0, q'(r_2)>0, q'(r_3)<0$. In terms of $T$ it means \[T(r_1,q_0)<0, T(r_2,q_0)>0, T(r_3,q_0)<0.\] \cref{EigenfunctionExpansion} and \cref{T r y} imply that \[\begin{aligned}
& \lim _{r \rightarrow 0} T(r, q_0)=+\infty, \\
& \lim _{r \rightarrow R} T(r, q_0)=+\infty.
\end{aligned}\]
Therefore, $T(\cdot,q_0)$ must change sign at least four times. Consequently, there are at least three zeros for $T'(r,q_0)= Z_{q_0}(r)$. Taking one more derivative, $Z'$ has two zeros. At one of these two zeros, $Z''\leq 0$. A Contradiction to \cref{PropertiesOfZy}(a). We complete the proof of \cref{RequirementsForQ}, and hence \cref{RequirementsOfG'/G}, \cref{Monotonicity Lemma} and the main theorem.
\end{proof}

A by-product of \cref{RequirementsForQ} is the following estimate on the first two Dirichlet eigenvalues of geodesic balls:
\begin{prop}
    The first two Dirichlet eigenvalues of geodesic balls in $\KK\HH^n$ satisfy:
    \[\frac{\lambda_2}{kn+2}-\frac{\lambda_1}{kn}\geq -\frac{2kn+3k-1}{3(kn+2)}.\]
\end{prop}
\begin{proof}
    By \cref{EigenfunctionExpansion}, we compute the expansion of $q$ at $r=0$:
    \[q(r) \sim 1+\left(2(\frac{-\lambda_2}{2kn+4}-\frac{\frac{kn-1}{3}+k-1}{kn+2})+\frac{\lambda_1}{kn}\right)r^2+o(r^2).\]
    Necessarily \[2(\frac{-\lambda_2}{2kn+4}-\frac{\frac{kn-1}{3}+k-1}{kn+2})+\frac{\lambda_1}{kn}\leq 0,\] which is equivalent to the statement.
\end{proof}
\begin{remark}
    In \cite[Theorem 3.1]{ashbaugh2001sharp} Ashbaugh and Benguira showed $\frac{\lambda_2-n}{\lambda_1} \geq \frac{n+2}{n}$ for balls in hemisphere. There they have used a recursive relation between the first 3 eigenfunctions, which we cannot prove in our case. They use this inequality to derive the properties of $q$, and then prove the PPW conjecture. We reversely prove the PPW conjecture first and then derive this inequality as a consequence.
\end{remark}

\appendix
\section{Proof of \cref{Crossproduct positivity lemma}}

In this appendix we prove \cref{Crossproduct positivity lemma}.
\par (1) We first show $((kn-1)^2\vec A_1+(k-1)^2\vec B_1)\times \vec A_2>0$: one can show that
\begin{align*}
    ((kn-1)^2\vec A_1+(k-1)^2\vec B_1)\times \vec A_2  &=\frac{2}{r^4 \cosh ^4 r \sinh ^4 r}[(k n-1)^2 \cosh ^4 r(2 r(\cosh 2 r+2)-3 \sinh 2 r)\\
    &-(k-1)^2 \sinh ^4 r \cdot(2 r(\cosh 2 r-2)-3\sinh{2r})].
\end{align*}
Benguria $\&$ Linde showed that $2r(\cosh 2 r+2)-3 \sinh 2 r>0$. Now observe that $kn-1>k-1$, $\cosh{r}>\sinh{r}$, and
\[2r(\cosh 2 r+2)-3 \sinh 2 r>2 r(\cosh 2 r-2)-3\sinh{2r},\]
hence the first term in the bracket above is greater than the second term. This proves  $(\vec A_1+\vec B_1)\times \vec A_2 $ for all $r>0, k=2,4,8,n>1.$
\par (2) Next we show $((kn-1)^2\vec A_1+(k-1)^2\vec B_1)\times ((kn-1)\vec A_4+(k-1)\vec B_4) >0:$ we calculate
\begin{align*}
    \text{The above cross product} &= 4(k-1)(k n-1)^2 r(2+\cosh 2 r) \operatorname{csch}^4 r \sech^2 r \\
    &+4(k-1)^2(k n-1) r(\cosh 2 r-2) \operatorname{csch}^2 r \operatorname{sech}^4 r \\
    &+\frac{2(k n-1)^3}{\sinh^6 r}  (2 r \cosh 2 r-\sin 2 r)+2(k-1)^3 \operatorname{sech}^6 r (2 r \cosh 2 r-\sin 2 r)\\
    &+\frac{2(k n-1)(k-1)}{\operatorname{sinh}^3 r \operatorname{cosh}^3 r } \left(k(n-1)\left(1-8 r^2\right)+(k n+k-2)\left(\cosh 2 r+4 r \sinh 2 r\right)\right)
\end{align*}
The sum of first two line is positive because $kn-1>n-1$, $\cosh r>\sinh{r}$ and $2+\cosh{2r}>2-\cosh{2r}$. The third line is positive because \cite{benguria2007second} showed $ 2 r \cosh 2 r-\sin 2 r>0$. We now show the last line is positive. We call \[f(r):=k(n-1)\left(1-8 r^2\right)+(k n+k-2)\left(\cosh 2 r+4 r \sinh 2 r\right).\] So $f(0)=2kn-2>0.$
Note that 
\[f(r)= 4r[(k n+k-2) \sinh{2r}-2(k n-k) r] +\text{ positive terms}.\]
Call $g(r):=(k n+k-2) \sinh{2r}-2(k n-k) r.$ Then $g(0)= 0$ and
\begin{align*}
    g'(r)&= 2(kn+k-2)\cosh{2r}-2(k n-k)\\
     &\geq 2(kn+k-2)-2(k n-k) \\
     &= 4k-4 >0.
\end{align*}
Hence $g>0$, and $f>0$ as we wish.
\par (3) To show $((kn-1)^2\vec A_1+(k-1)^2\vec B_1)\times ((kn-1)\vec A_5+(k-1)\vec B_5)>0$, we use induction on $n$. We here present the case $k=4$, i.e. $\HH\HH^n$. The other cases are similar. First, one verifies that in the base case $n=2$, 
\begin{align*}
    (49\vec A_1+9\vec B_1)\times (7\vec A_5+3\vec B_5) &= \frac{1}{16\sinh^8{r}\cosh^8{r}}\big\lbrack-7000 r-12292 r \operatorname{cosh}[2 r]-10400 r \operatorname{cosh}[4 r]\\
    &-2734 r \operatorname{cosh}[6 r]-360 r \operatorname{cosh}[8r]-142 r \operatorname{cosh}[10 r]+2868 \operatorname{sinh}[2 r]\\
    &+2985 \operatorname{sinh}[4 r]+1252 \operatorname{sinh}[6 r]+720 \operatorname{sinh}[8 r]+192 \operatorname{sinh}[10 r]+5 \operatorname{sinh}[12r]\big\rbrack.
\end{align*}
The Taylor expansion of the numerator at $r=0$ shows that each coefficient is positive, so the function is positive for all $r>0$.
The case $n=3$ can be computed as
\begin{align*}
    (121\vec A_1+9\vec B_1)\times (11\vec A_5+3\vec B_5)&= (49\vec A_1+9\vec B_1)\times (7\vec A_5+3\vec B_5)\\
    &+ (11-7)(49\vec A_1+9\vec B_1)\times \vec A_5\\
    &+ (121-49)\vec A_1 \times (7\vec{A_5}+3\vec{B_5})\\
    &+ (121-49)(11-7)\vec A_1\times \vec A_5.
\end{align*}
The first line is the base case, and the last line was shown to be positive in \cite{benguria2007second}. For the third line,
\begin{align*}
   \vec A_1 \times (7\vec{A_5}+3\vec{B_5}) =\frac{1}{8\sinh^8{r}\cosh^4{r}}&(136 r-496 r \cosh{2r}+56r \operatorname{cosh}4r\\
   &-32 r \operatorname{cosh}6r+150 \operatorname{sinh}2r\\
   &-50 \operatorname{sinh}4r+38 \operatorname{sinh}6r+\operatorname{sinh}8r). 
\end{align*}
Call $f(r):= \text{Numerator of R.H.S.}$, then $f(0)=\cdots =f^{(4)}(0)=0$. By using Taylor expansion at $r=0$ one checks that each power of $r$ has positive coefficients. Since hyperbolic trigonometric functions are analytic functions, it shows that $f(r)>0$. Similarly the second line is also positive.
As $n$ increases, only the coefficients in the last three lines increase. Hence the whole expression remains positive for all $n>1$.

\par (4) Now we show $((kn-1)^2\vec A_1+(k-1)^2\vec B_1)\times ((kn-1)^2\vec A_3+(k-1)^2\vec B_3) >0$. We take $\HH\HH^n$, i.e. $k=4$ as an example. The other non-compact ROSS can be shown similarly. As above, we use induction. The base case $n=2$ is done as follows:
\begin{align*}
    &f(x)= \frac{1}{4}\sinh^4r\cosh^4 r(49\vec A_1+9\vec B_1)\times (49\vec A_3+9\vec B_3) \\
    &= -1764 x^2 \operatorname{Coth}[x]+2401 x \operatorname{Coth}[x]^4\\
    &+49 \operatorname{Cosh}[x]^2\left(90 x-72 x^2 \operatorname{Coth}[x]+49 \operatorname{Coth}[x]^3-294 x \operatorname{Coth}[x]^4+196 x^2 \operatorname{Coth}[x]^5\right)\\
    &-9\big[x\left(-98+196 x \operatorname{Tanh}[x]-9 \operatorname{Tanh}[x]^4\right)\\
    &+\operatorname{Sinh}[x]^2\left(490 x-392 x^2 \operatorname{Tanh}[x]+9 \operatorname{Tanh}[x]^3-54 x \operatorname{Tanh}[x]^4+36 x^2 \operatorname{Tanh}[x]^5\right)\big]
\end{align*}
\begin{itemize}
    \item For $0<x \leq 1$, \[f(x)\sim \frac{76832 x}{45}-\frac{551936 x^3}{135}+\frac{2609152 x^5}{675}+\frac{7491328 x^7}{22275}+...> \frac{1}{45}x>0.\]
    \item For $1<x<r_0\approx 1.4,$ see \cref{(49A1+9B1)times(49A3+9B3)}
    \begin{figure}
    \centering
    \includegraphics[width=0.5\linewidth]{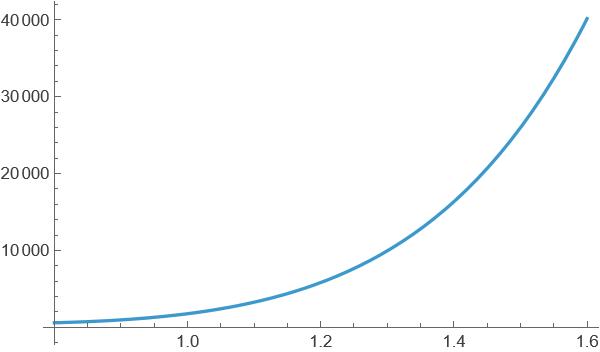}
    \caption{$(49A1+9B1)\times(49A3+9B3)$}
    \label{(49A1+9B1)times(49A3+9B3)}
\end{figure}
\end{itemize}
The case $n=3$ is given by
\begin{align*}
    (121\vec A_1+9\vec B_1)\times (121\vec A_3+9\vec B_3)&= (49\vec A_1+9\vec B_1)\times (49\vec A_3+9\vec B_3)\\
    &+ (121-49)(\vec{A}_1 \times(44 \vec{A}_3+\vec B_3))\\
    &+ (121-49)((55 \vec A_1+9 \vec B_1) \times \vec A_3).
\end{align*}
The first line is the base case. For the second line,
\begin{align*}
    \sinh^6 r (\vec{A}_1 \times(44 \vec{A}_3+\vec B_3)) &= \textcolor{red}{176 \coth r-1056r \coth^2r+704 r^2 \coth^3 r}+\frac{176}{\sinh^2 r}\\
    &+\frac{\sinh^2 r}{\cosh^5 r}(\textcolor{blue}{18r \cosh r-6r \cosh{3r}+\sinh r -44r^2\sinh r+ \sinh{3r}+4r^2\sin{3r}}).
\end{align*}
We deal with the red and blue parts separately and show they are both positive. For the red terms,
\begin{itemize}
    \item If $0<r \leq  1$, Taylor expansion at $r=0$ shows
    \[\sinh^6 r (\vec{A}_1 \times(44 \vec{A}_3+\vec B_3))\sim \frac{6352 x^3}{45}-\frac{22384 x^5}{1315}+\frac{25040 x^7}{189}> \frac{1}{45}x^3\] hence is positive.
    \item If $r\geq 2$, $1<\coth{r}<1.05$ and $e^{-r}< e^{-2}<0.15$, then 
    \[\textcolor{red}{red \quad terms} \geq 176-1100r+704r^2 >0\] and 
    \begin{align*}
        \textcolor{blue}{blue \quad terms} &= 4r(e^r+e^{-r})-3r(e^{3r}+e^{-3r})+\frac{1-44r^2}{2}(e^r+e^{-r})+\frac{1+4r^2}{2}(e^{3r}+e^{-3r})\\
        & \geq 4re^r-3r(e^{3r}+0.01)+\frac{1-44r^2}{2}(e^r+0.15)+\frac{1+4r^2}{2}e^{3r} > 0
    \end{align*}
    \item If $1<r<2$, see \cref{Sinh[x]^6 (A1 x (44A3+B3))}
    \begin{figure}
    \centering
    \includegraphics[width=0.5\linewidth]{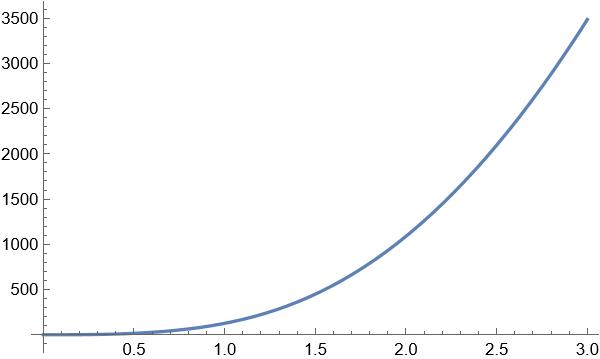}
    \caption{$\sinh^6r (\vec{A_1}\times(44\vec{A_3}+\vec{B_3}))$ }
    \label{Sinh[x]^6 (A1 x (44A3+B3))}
\end{figure}
\end{itemize}
\par The third line can be proved similarly.
\par As $n$ increases, only the coefficients $(121-49)$ increases. Hence we proved for all $n>1$ the expression is positive.
\par (5) For $(kn-1)^2 \vec A_1+(k-1)^2 \vec B_1)\times \vec B_2$, the situation is more complicated. This term does not have a constant sign. The diagram for the case $\HH\HH^2$ is shown in \cref{(49A1+9B1)xB2}.
\begin{figure}
    \centering
    \includegraphics[width=0.5\linewidth]{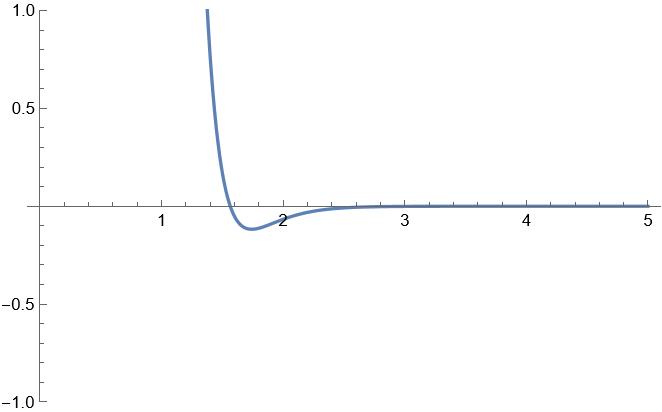}
    \caption{$(49\vec A_1 + 9\vec B_1)\times \vec B_2$}
    \label{(49A1+9B1)xB2}
\end{figure}
There is a root at approximately $r_1(k=2, n=2)\approx 1.57$. When $r<r_1(k,n)$, this term is positive.
\par (6) For $\vec B_2 \times \vec A_2$, the diagram is shown in \cref{B2xA2}.
\begin{figure}
    \centering
    \includegraphics[width=0.5\linewidth]{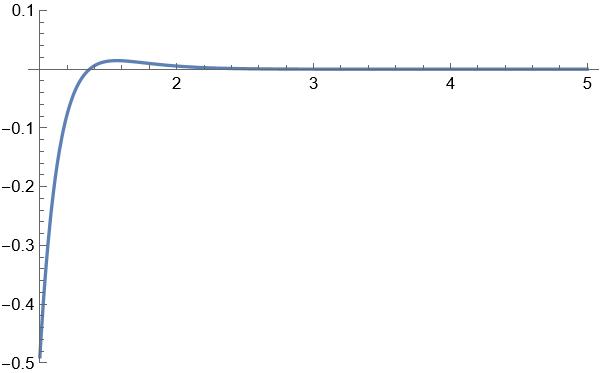}
    \caption{$\vec B_2 \times \vec A_2$}
    \label{B2xA2}
\end{figure}
There is a root $r_2 \approx 1.35$. 
\begin{claim}
    The root of $\vec B_2 \times \vec A_2$,  $r_2$, is less than $r_1(k,n)$ above.
\end{claim}
The proof of the claim is follows: Using numerical approximation one can compute that at $r=1.4>r_2$, the value of  $((kn-1)^2\vec A_1+(k-1)^2\vec B_1)\times ((kn-1)^2\vec A_3+(k-1)^2\vec B_3)$ is approximately given by 
\[f(k,n):=0.03293-0.0423419 k+0.0211709 k^2-0.0235181 kn+0.0117591kn^2.\]
As a function of $n \in \RR$, $f(2,n)$ and $f(4,n)$ are both quadratic. One can show they are both increasing when $n\geq 2$, and are both positive at $n=2$. The situation for $\OO\HH^2$ is similar. That is, up to $r=1.4$ the value of $((kn-1)^2\vec A_1+(k-1)^2\vec B_1)\times ((kn-1)^2\vec A_3+(k-1)^2\vec B_3)$ is not yet negative, so $r_1(k,n)>1.4$.
\par Fix $r_0(k,n)= \max\{r_1(k,n), r_2\} =r_1(k,n) $. We choose $(kn-1)^2\vec A_1+(k-1)^2\vec B_1$ as the reference vector when $r<r_0$; otherwise we choose $\vec B_2$.

\par (7) For $\vec B_2 \times ((kn-1)^2\vec A_3+(k-1)^2\vec B_3)$, we only need the positivity when $r>r_0\approx 1.4$. we first calculate  
\[ \vec B_2 \times \vec A_3= \frac{r^2}{4\sinh{^9 r\cosh^4 r}}  \left(\left(-3+4 r^2\right) \operatorname{cosh}r+3\left(1+4 r^2\right) \operatorname{cosh}3r+2 r(\operatorname{sinh}r-5 \operatorname{sinh}3r)\right)\]
The Taylor expansion of the numerator has all positive coefficients, so it is positive.
 The base case $\vec B_2 \times ((2k-1)^2\vec A_3+(k-1)^2\vec B_3)>0$ is verified as follows. We only demonstrate the case $k=4$ and others are similar.
 \begin{align*}
     f(r)&= \sinh^9{r}\cosh^9{r}(\vec B_2 \times ((2*4-1)^2\vec A_3+(4-1)^2\vec B_3))\\
     &= \frac{r^2}{16}\big\lbrack-75+500r^2+87(-1+16r^2)\cosh{2r}+60\left(1+12 r^2\right) \operatorname{cosh}{4r}+(87+464r^2) \operatorname{cosh}{6r}\\
     &+(15+60r^2)\cosh{8r}-580 r \operatorname{sinh}{2r}-460 r \operatorname{sinh}{4r}-348 r \operatorname{sinh}{6 r}-50 r \operatorname{sinh}{8r}\big\rbrack.
 \end{align*}
As $r\geq r_0> 1$, $500r^2-75>0$, and the coefficients of $\cosh{2r}, \cosh{4r}, \cosh{6r}$ and $\cosh{8r}$ is greater than the coefficients of corresponding $\sinh$-terms respectively. The $\sinh$-terms are dominated by the $\cosh$-terms. Hence $f(r)>0$ as we wish.

 We conclude that $f(r)>0$ for all $r>r_0$ and therefore the base case is proved. Now as $n$ increases, we are putting more $\vec B_2 \times \vec A_3$ terms. Hence the whole term is positive for all $n>1$.
\par (8) For $\vec B_2 \times ((kn-1)\vec A_i+(k-1)\vec B_i)$ where $i=4,5$, we use induction on $n$. Positivity of the base case $\vec B_2 \times ((2k-1)\vec A_i+(k-1)\vec B_i)$ can be proved as follows (taking $k=4$ as an example):
\[\vec B_2 \times ((2*4-1)\vec A_4+(4-1)\vec B_4)=2 r^2 \csch^4 {r} \sech^4r(-6r\cosh{2r}+(3+4r^2)\sinh{2r})\]
The terms in the bracket 
\[-6r\cosh{2r}+(3+4r^2)\sinh{2r}>0 \text{ for all } r>0\] because each coefficient in the Taylor expansion at $r=0$ is positive.
\par The case $\vec B_2 \times ((2*4-1)\vec A_5+(4-1)\vec B_5)>0 \text{ for } r>r_0 \approx 1.4$ can be done as (7). 
\par To deal with $\vec B_2 \times \vec A_i$ terms we do Taylor expansion as (3). Increasing dimension is equivalent to putting more $\vec B_2 \times \vec A_i$ terms.
 
\par We finally completes the proof of \cref{Crossproduct positivity lemma}.

\printbibliography
\end{document}